\documentclass{amsart}
\usepackage{setspace}
\usepackage{a4}
\usepackage{amsthm}
\usepackage{latexsym}
\usepackage{amsfonts}
\usepackage{graphicx}
\usepackage{textcomp}
\usepackage{cite}
\usepackage{enumerate}
\usepackage{amssymb}
\usepackage{hyperref}
\usepackage{amsmath}
\usepackage{tikz}
\usepackage[mathscr]{euscript}
\usepackage{mathtools}
\newtheorem{theorem}{Theorem}[section]

\newtheorem{corollary}[theorem] {Corollary}
\newtheorem{definition}[theorem]{Definition}

\newtheorem{remark}[theorem]{Remark}

\title{This is the title}
\usepackage{fancyhdr}

\begin{document}
\hrule\hrule\hrule\hrule\hrule
\vspace{0.3cm}	
\begin{center}
{\bf{CONTINUOUS  KRISHNA-PARTHASARATHY ENTROPIC UNCERTAINTY PRINCIPLE}}\\
\vspace{0.3cm}
\hrule\hrule\hrule\hrule\hrule
\vspace{0.3cm}
Dedicated to the memory of Prof. K. R. Parthasarathy
\vspace{0.3cm}
\hrule\hrule\hrule
\vspace{0.3cm}
\textbf{K. MAHESH KRISHNA}\\
Post Doctoral Fellow \\
Statistics and Mathematics Unit\\
Indian Statistical Institute, Bangalore Centre\\
Karnataka 560 059, India\\
Email: kmaheshak@gmail.com\\

Date: \today
\end{center}

\hrule\hrule
\vspace{0.5cm}
%--------------------------------------
\textbf{Abstract}: In 2002, Krishna and Parthasarathy [\textit{Sankhy\={a} Ser. A}] derived discrete quantum  version of Maassen-Uffink [\textit{Phys. Rev. Lett., 1988}] entropic uncertainty principle. In this paper, using the notion of continuous operator-valued frames, we derive an entropic uncertainty principle for arbitrary family of operators indexed by   measure spaces having finite measure. We give an application to the special case of compact groups.



\hrule

%\tableofcontents
\hrule
\section{Introduction}
Let $\mathcal{H}$ be a finite dimensional Hilbert space. Given an orthonormal basis  $\{\tau_j\}_{j=1}^n$ for $\mathcal{H}$, the \textbf{(finite) Shannon entropy}  at a point $h \in \mathcal{H}_\tau$ is defined as 
\begin{align}\label{D}
	S_\tau (h)\coloneqq - \sum_{j=1}^{n} \left|\left \langle h, \tau_j\right\rangle \right|^2\log \left|\left \langle h, \tau_j\right\rangle \right|^2,
\end{align}
where $\mathcal{H}_\tau\coloneqq \{h \in \mathcal{H}: \|h\|=1,  \langle h , \tau_j \rangle \neq 0, 1\leq j \leq n\}$ \cite{DEUTSCH}. In 1983, Deutsch derived following breakthrough entropic uncertainty principle for Shannon entropy  \cite{DEUTSCH}.
\begin{theorem}\cite{DEUTSCH} (\textbf{Deutsch Entropic Uncertainty Principle})  \label{DU}
Let $\{\tau_j\}_{j=1}^n$,  $\{\omega_j\}_{j=1}^n$ be two orthonormal bases for a  finite dimensional Hilbert space $\mathcal{H}$. Then 
	\begin{align*}
2 \log n \geq S_\tau (h)+S_\omega (h)\geq -2 \log \left(\frac{1+\displaystyle \max_{1\leq j, k \leq n}|\langle\tau_j , \omega_k\rangle|}{2}\right), \quad \forall h \in \mathcal{H}_\tau\cap \mathcal{H}_\omega.
	\end{align*}
\end{theorem}
In 1988, Maassen  and Uffink (motivated from the conjecture  by Kraus made in 1987 \cite{KRAUS}) improved Deutsch entropic uncertainty principle.
\begin{theorem}\cite{MAASSENUFFINK} 
(\textbf{Maassen-Uffink Entropic Uncertainty Principle})  \label{MU}
Let $\{\tau_j\}_{j=1}^n$,  $\{\omega_j\}_{j=1}^n$ be two orthonormal bases for a  finite dimensional Hilbert space $\mathcal{H}$. Then 
\begin{align*}
	2 \log n \geq S_\tau (h)+S_\omega (h)\geq -2 \log \left(\displaystyle \max_{1\leq j, k \leq n}|\langle\tau_j , \omega_k\rangle|\right), \quad \forall h \in \mathcal{H}_\tau \cap \mathcal{H}_\omega.
\end{align*}	
\end{theorem}

In 2002, motivated from the theory of quantum computation and quantum information, Krishna and Parthasarathy improved Maassen-Uffink entropic uncertainty principle \cite{KRISHNAPARTHASARATHY}. To do so, first they introduced the notion of entropy for quantum measurements.  Let  $\{P_j\}_{j=1}^n$ be a collection of  orthogonal projections on a finite dimensional Hilbert space $\mathcal{H}$ satisfying 
\begin{align*}
	\sum_{j=1}^{n}P_j=I_\mathcal{H},
\end{align*}
the identity operator on $\mathcal{H}$. The Shannon entropy of $\{P_j\}_{j=1}^n$ at a point  $h \in \mathcal{H}_P$  is defined as 
\begin{align}\label{KP}
	S_P(h)\coloneqq -\sum_{j=1}^{n}  \langle P_jh, h\rangle \log \langle P_jh, h\rangle , 
\end{align}
where $\mathcal{H}_P\coloneqq \{h \in \mathcal{H}: \|h\|=1,  P_jh \neq 0, 1\leq j \leq n\}$. Note that Equation (\ref{KP}) reduces to Equation (\ref{D}) whenever $P_j$ is the projection onto the span of  $\tau_j$ for each $j$.  Krishna and Parthasarathy made the following improvement of Theorem \ref{MU}.

\begin{theorem}\cite{KRISHNAPARTHASARATHY}   (\textbf{Krishna-Parthasarathy Entropic Uncertainty Principle})  \label{KPU}
Let $\{P_j\}_{j=1}^n$,  $\{Q_k\}_{k=1}^m$ be two collections of orthogonal projections  on  a  finite dimensional Hilbert space $\mathcal{H}$ such that 
\begin{align*}
	\sum_{j=1}^{n}P_j=I_\mathcal{H}=
	\sum_{k=1}^{m}Q_k.
\end{align*}
Then 
\begin{align*}
2 \log n \geq S_P (h)+S_Q (h)&\geq -2 \log \left(\displaystyle \max_{1\leq j \leq n, 1\leq k \leq m}\frac{|\langle P_jQ_kh, h\rangle|}{\|P_j h\|\|Q_kh\|}\right)	\\
&\geq 2 \log \left(\displaystyle \max_{1\leq j \leq n, 1\leq k \leq m}\|P_jQ_k\|\right), \quad \forall h \in \mathcal{H}_P \cap \mathcal{H}_Q.
\end{align*}	
\end{theorem}
Krishna and Parthasarathy were able to give a beautiful application of Theorem \ref{KPU} to finite groups using Peter-Weyl theorem \cite{KRISHNAPARTHASARATHY}. They also extended Theorem \ref{KPU} to positive operators. We are motivated by the following two questions.
\begin{enumerate}
	\item What is the version of  Theorem \ref{KPU} for collection of operators which are not necessarily positive?
	\item What is the version of Theorem  \ref{KPU} indexed by measure spaces?
\end{enumerate}
 Using the theory of frames, we are going to do both of  these in the paper. We also give an application of our theorem to compact groups using Peter-Weyl representation theory.

\section{Continuous   Krishna-Parthasarathy  Entropic Uncertainty Principle}
In the paper,   $\mathbb{K}$ denotes $\mathbb{C}$ or $\mathbb{R}$ and $\mathcal{H}$  denotes a Hilbert space  (need not be finite dimensional) over $\mathbb{K}$. Given two Hilbert spaces $\mathcal{H}$ and  $\mathcal{H}_0$, the set of all continuous linear operators from $\mathcal{H}$ to  $\mathcal{H}_0$ is denoted by $\mathcal{B}(\mathcal{H}, \mathcal{H}_0)$. Given a measure space $(\Omega, \mu)$ and a Hilbert space  $\mathcal{H}$,  recall that 
\begin{align*}
	\mathcal{L}^2(\Omega, \mathcal{H})\coloneqq \left\{f:\Omega \to  \mathcal{H} \text{  is weakly measurable}, \int\limits_{\Omega}\|f(\alpha)\|^2\,d\mu(\alpha)<\infty\right\}.
\end{align*}
 Unlike the discrete case, we need the notion of frames to handle continuous case. 
\begin{definition}\cite{ABDOLLAHPOURFAROUGHI, SUN, KAFTALLARSONZHANG}
	Let $(\Omega, \mu)$ be a measure space and let $\mathcal{H}$, $\mathcal{H}_0$ be Hilbert spaces. A family    $\{A_\alpha\}_{\alpha\in \Omega}$ of continuous linear operators in $\mathcal{B}(\mathcal{H}, \mathcal{H}_0)$  is said to be a \textbf{continuous operator-valued Parseval frame}  in $\mathcal{B}(\mathcal{H}, \mathcal{H}_0)$  if the following conditions hold.
	\begin{enumerate}[\upshape(i)]
		\item For each $h \in \mathcal{H}$, the map $\Omega \ni \alpha \mapsto A_\alpha h  \in \mathcal{H}_0$ is weakly measurable.
		\item 
		\begin{align*}
			\|h\|^2=\int\limits_{\Omega}\|A_\alpha h\|^2\,d\mu(\alpha), \quad \forall h \in \mathcal{H}.
		\end{align*}
	\end{enumerate}
\end{definition}
To proceed we need to generalize Definition \ref{KP}.
\begin{definition}
Let $(\Omega, \mu)$ be a measure space and let $\mathcal{H}$, $\mathcal{H}_0$ be Hilbert spaces. Given a continuous  operator-valued Parseval frame     $\{A_\alpha\}_{\alpha\in \Omega}$  in $\mathcal{B}(\mathcal{H}, \mathcal{H}_0)$, we define the continuous Shannon entropy  	of $\{A_\alpha\}_{\alpha\in \Omega}$ at a point  $h \in \mathcal{H}_A$   as 
\begin{align}\label{OKP}
	S_A(h)\coloneqq -\int\limits_{\Omega}\|A_\alpha h\|^2\log \|A_\alpha h \|^2\, d\mu(\alpha), 
\end{align}
where $\mathcal{H}_A\coloneqq \{h \in \mathcal{H}: \|h\|=1,  A_\alpha h \neq 0, \forall \alpha \in \Omega\}$. 
\end{definition}
Before deriving main theorem, we recall Riesz-Thorin interpolation theorem which we are going to use. 
\begin{theorem}\cite{STEINSHAKARCHI} (\textbf{Riesz-Thorin Interpolation Theorem}) \label{RT}
Let $(\Omega, \mu)$	and $(\Delta, \nu)$ be measure spaces and let $\mathcal{X}$ be a Banach space. Let $1\leq p_0, q_0, p_1. q_1\leq \infty$. Suppose 
\begin{align*}
	T:	\mathcal{L}^{p_0}(\Omega, \mathcal{X})+	\mathcal{L}^{p_1}(\Omega, \mathcal{X})\to \mathcal{L}^{q_0}(\Delta, \mathcal{X})+	\mathcal{L}^{q_1}(\Delta, \mathcal{X})
\end{align*}
is a linear operator such that both 
\begin{align*}
&T:	\mathcal{L}^{p_0}(\Omega, \mathcal{X})\to \mathcal{L}^{q_0}(\Delta, \mathcal{X})\\
&T:	\mathcal{L}^{p_1}(\Omega, \mathcal{X})\to \mathcal{L}^{q_1}(\Delta, \mathcal{X})
\end{align*}
are bounded linear operators. For $0<r<1$, define $p_r$ and $q_r$ as 
\begin{align*}
	\frac{1}{p_r}\coloneqq \frac{1-r}{p_0}+\frac{r}{p_1},  \quad \frac{1}{q_r}\coloneqq \frac{1-r}{q_0}+\frac{r}{q_1}.
\end{align*}
Then 
\begin{align*}
	T:	\mathcal{L}^{p_r}(\Omega, \mathcal{X})\to \mathcal{L}^{q_r}(\Delta, \mathcal{X})
\end{align*}
is a bounded linear operator  and 
\begin{align*}
	\|T\|_{	\mathcal{L}^{p_r}(\Omega, \mathcal{X})\to \mathcal{L}^{q_r}(\Delta, \mathcal{X})}\leq 	\|T\|_{	\mathcal{L}^{p_0}(\Omega, \mathcal{X})\to \mathcal{L}^{q_0}(\Delta, \mathcal{X})}^{1-r}
	\|T\|_{	\mathcal{L}^{p_1}(\Omega, \mathcal{X})\to \mathcal{L}^{q_1}(\Delta, \mathcal{X})}^r.
\end{align*}
\end{theorem}
Following is the main theorem of this article which we would like to call continuous Krishna-Parthasarathy entropic uncertainty principle. 
\begin{theorem}\label{UPG}
Let $(\Omega, \mu)$	and $(\Delta, \nu)$ be finite measure spaces and let 	$\mathcal{H}$, $\mathcal{H}_0$ be Hilbert spaces.  Let $\{A_\alpha\}_{\alpha\in \Omega}$ and  $\{B_\beta\}_{\beta \in \Delta}$ be  continuous  operator-valued Parseval frames  in $\mathcal{B}(\mathcal{H}, \mathcal{H}_0)$. Then 
\begin{align*}
	 \log ((\mu(\Omega)\nu(\Delta)) \geq S_A (h)+S_B (h)\geq -2 \log \left(\displaystyle \sup_{\alpha \in \Omega, \beta \in \Delta}\|B_\beta A_\alpha^*\|\right), \quad \forall h \in \mathcal{H}_A \cap \mathcal{H}_B.	
\end{align*}
\end{theorem}
\begin{proof}
Define 
\begin{align*}
	T:	\mathcal{L}^{2}(\Omega, \mathcal{H}_0)+	\mathcal{L}^{1}(\Omega, \mathcal{H}_0) \to  \mathcal{L}^{2}(\Delta, \mathcal{H}_0)+	\mathcal{L}^{\infty}(\Delta, \mathcal{H}_0)
\end{align*}
by 
\begin{align*}
	&T:	\mathcal{L}^{2}(\Omega, \mathcal{H}_0) \ni f \mapsto Tf \in \mathcal{L}^{2}(\Delta, \mathcal{H}_0); \quad (Tf)(\beta)\coloneqq \int\limits_{\Omega}B_\beta A_\alpha^*(f(\alpha))\,d\mu(\alpha), \quad \forall \beta \in \Delta,\\
	&T:	\mathcal{L}^{1}(\Omega, \mathcal{H}_0)\ni g \mapsto Tg \in \mathcal{L}^{\infty}(\Delta, \mathcal{H}_0); \quad (Tg)(\beta)\coloneqq \int\limits_{\Omega}B_\beta A_\alpha^*(g(\alpha))\,d\mu(\alpha), \quad \forall \beta \in \Delta.
\end{align*}	
Let $f \in \mathcal{L}^{2}(\Omega, \mathcal{H}_0)$. Then 
\begin{align*}
	\|Tf\|^2_{\mathcal{L}^{2}(\Delta, \mathcal{H}_0)}&=\int\limits_{\Delta}\|(Tf)(\beta)\|^2\,d\nu(\beta)=\int\limits_{\Delta}\left\|\int\limits_{\Omega}B_\beta A_\alpha^*(f(\alpha))\,d\mu(\alpha)\right\|^2\,d\nu(\beta)\\
	&=\int\limits_{\Delta}\left\|B_\beta\left(\int\limits_{\Omega} A_\alpha^*(f(\alpha))\,d\mu(\alpha)\right)\right\|^2\,d\nu(\beta)=\left\|\int\limits_{\Omega} A_\alpha^*(f(\alpha))\,d\mu(\alpha)\right\|^2\\
	&=\sup_{h \in \mathcal{H}, \|h\|\leq1}\left|\left\langle \int\limits_{\Omega} A_\alpha^*(f(\alpha))\,d\mu(\alpha), h\right\rangle \right|^2=\sup_{h \in \mathcal{H}, \|h\|\leq1}\left|\ \int\limits_{\Omega} \langle f(\alpha), A_\alpha h\rangle \,d\mu(\alpha)\right|^2\\
	&\leq \sup_{h \in \mathcal{H}, \|h\|\leq1}\left(\ \int\limits_{\Omega} \|f(\alpha)\|^2 \,d\mu(\alpha)\right)\left(\ \int\limits_{\Omega} \|A_\alpha h\|^2 \,d\mu(\alpha)\right)=\|f\|_{\mathcal{L}^{2}(\Omega, \mathcal{H}_0)}^2.
\end{align*}
Let $g \in \mathcal{L}^{1}(\Omega, \mathcal{H}_0)$. Then for any $\beta \in \Delta$, we have 
\begin{align*}
	\|(Tg)(\beta)\|&=\left\|\int\limits_{\Omega}B_\beta A_\alpha^*(g(\alpha))\,d\mu(\alpha)\right\|\leq \int\limits_{\Omega}\|B_\beta A_\alpha^*(g(\alpha))\|\,d\mu(\alpha)\\
	&\leq \int\limits_{\Omega}\|B_\beta A_\alpha^*\|\|(g(\alpha))\|\,d\mu(\alpha)\leq \left(\displaystyle \sup_{\alpha \in \Omega, \beta \in \Delta}\|B_\beta A_\alpha^*\|\right)\int\limits_{\Omega}\|(g(\alpha))\|\,d\mu(\alpha)\\
	&=\left(\displaystyle \sup_{\alpha \in \Omega, \beta \in \Delta}\|B_\beta A_\alpha^*\|\right)	\|g\|^2_{\mathcal{L}^{1}(\Delta, \mathcal{H}_0)}.
\end{align*}
Hence 
\begin{align*}
		\|Tg\|_{\mathcal{L}^{\infty}(\Delta, \mathcal{H}_0)}\leq \left(\displaystyle \sup_{\alpha \in \Omega, \beta \in \Delta}\|B_\beta A_\alpha^*\|\right)	\|g\|^2_{\mathcal{L}^{1}(\Delta, \mathcal{H}_0)}.
\end{align*}
Define 
\begin{align*}
	M\coloneqq \displaystyle \sup_{\alpha \in \Omega, \beta \in \Delta}\|B_\beta A_\alpha^*\|.
\end{align*}
Now applying Theorem \ref{RT} for the case 
\begin{align*}
	p_0=q_0=2, \quad p_1=1, \quad p_2=\infty
\end{align*}
we get that the operator 
\begin{align*}
	T:	\mathcal{L}^{p_r}(\Omega, \mathcal{H}_0)\to \mathcal{L}^{q_r}(\Delta, \mathcal{H}_0)
\end{align*}
is bounded for every $0<r<1$, where $p_r$ and $q_r$ are given by 
\begin{align*}
	&\frac{1}{p_r}\coloneqq \frac{1-r}{2}+\frac{r}{1}\implies p_r=\frac{2}{r+1}, \\
	&  \frac{1}{q_r}\coloneqq \frac{1-r}{2}+\frac{r}{\infty}\implies q_r=\frac{2}{1-r}.
\end{align*}
Therefore 
\begin{align*}
		T:	\mathcal{L}^{\frac{2}{r+1}}(\Omega, \mathcal{H}_0)\to \mathcal{L}^{\frac{2}{1-r}}(\Delta, \mathcal{H}_0)
	\end{align*}
is a bounded linear operator for every  $0<r<1$. Moreover, Theorem \ref{RT} also gives 
	\begin{align*}
		\|T\|_{	\mathcal{L}^{\frac{2}{r+1}}(\Omega, \mathcal{H}_0)\to \mathcal{L}^{\frac{2}{1-r}}(\Delta, \mathcal{H}_0)}\leq 	M^r, \quad \forall 0<r<1.
\end{align*}
i.e., 
	\begin{align*}
	\|Tf\|_{\mathcal{L}^{\frac{2}{1-r}}(\Delta, \mathcal{H}_0)}\leq 	M^r	\|f\|_{\mathcal{L}^{\frac{2}{r+1}}(\Omega, \mathcal{H}_0)}, \quad \forall 0<r<1, \forall f \in \mathcal{L}^{\frac{2}{r+1}}(\Omega, \mathcal{H}_0).
\end{align*}
Let $h \in \mathcal{H}_A \cap \mathcal{H}_B$ be fixed. Define 
\begin{align*}
	f:\Omega \ni \alpha \mapsto f(\alpha)\coloneqq A_\alpha h \in \mathcal{H}_0.
\end{align*}
Then $f \in 	\mathcal{L}^{2}(\Omega, \mathcal{H}_0)$. Since $\mu(\Omega)<\infty$ is a finite and $0<r<1$, we have $\frac{2}{r+1}<2$.
But then  $\mathcal{L}^{2}(\Omega, \mathcal{H}_0)\subseteq \mathcal{L}^{\frac{2}{r+1}}(\Omega, \mathcal{H}_0)$. Therefore  $f \in \mathcal{L}^{\frac{2}{r+1}}(\Omega, \mathcal{H}_0)$. We now find 
\begin{align*}
&\|Tf\|_{\mathcal{L}^{\frac{2}{1-r}}(\Delta, \mathcal{H}_0)}=\left(\int\limits_{\Delta}\|(Tf)(\beta)\|^{\frac{2}{1-r}}\,d\nu(\beta)\right)^\frac{1-r}{2}=\left(\int\limits_{\Delta}\left\| \int\limits_{\Omega}B_\beta A_\alpha^*(f(\alpha))\,d\mu(\alpha)\right\|^{\frac{2}{1-r}}\,d\nu(\beta)\right)^\frac{1-r}{2}\\
&=\left(\int\limits_{\Delta}\left\| \int\limits_{\Omega}B_\beta A_\alpha^*A_\alpha h\,d\mu(\alpha)\right\|^{\frac{2}{1-r}}\,d\nu(\beta)\right)^\frac{1-r}{2}
=\left(\int\limits_{\Delta}\left\|B_\beta\left(\int\limits_{\Omega} A_\alpha^* A_\alpha h\,d\mu(\alpha)\right)\right\|^{\frac{2}{1-r}}\,d\nu(\beta)\right)^\frac{1-r}{2}\\
&=\left(\int\limits_{\Delta}\left\|B_\beta h\right\|^{\frac{2}{1-r}}\,d\nu(\beta)\right)^\frac{1-r}{2}.
\end{align*} 
By substituting the above value in 
\begin{align*}
\|Tf\|_{\mathcal{L}^{\frac{2}{1-r}}(\Delta, \mathcal{H}_0)}\leq 	M^r	\|f\|_{\mathcal{L}^{\frac{2}{r+1}}(\Omega, \mathcal{H}_0)}, \quad \forall 0<r<1,	
\end{align*}
we get 
\begin{align*}
	\left(\int\limits_{\Delta}\left\|B_\beta h\right\|^{\frac{2}{1-r}}\,d\nu(\beta)\right)^\frac{1-r}{2}\leq M^r \left(\int\limits_{\Omega}\left\|A_\alpha h\right\|^{\frac{2}{r+1}}\,d\mu(\alpha)\right)^\frac{r+1}{2}, \quad \forall 0<r<1.
\end{align*}
Previous inequality gives
\begin{align*}
\left(\int\limits_{\Omega}\left\|A_\alpha h\right\|^{\frac{2}{r+1}}\,d\mu(\alpha)\right)^\frac{-r-1}{2}\left(\int\limits_{\Delta}\left\|B_\beta h\right\|^{\frac{2}{1-r}}\,d\nu(\beta)\right)^\frac{1-r}{2}\leq M^r, \quad \forall 0<r<1.	 
\end{align*}
By raising to the power $2/r$, we get 
\begin{align*}
	\left(\int\limits_{\Omega}\left\|A_\alpha h\right\|^{\frac{2}{r+1}}\,d\mu(\alpha)\right)^\frac{-r-1}{r}\left(\int\limits_{\Delta}\left\|B_\beta h\right\|^{\frac{2}{1-r}}\,d\nu(\beta)\right)^\frac{1-r}{r}\leq M^2, \quad \forall 0<r<1.	 
\end{align*}
Taking logarithm gives 
\begin{align*}
\frac{-r-1}{r} \log 	\left(\int\limits_{\Omega}\left\|A_\alpha h\right\|^{\frac{2}{r+1}}\,d\mu(\alpha)\right)+\frac{1-r}{r}\log \left(\int\limits_{\Delta}\left\|B_\beta h\right\|^{\frac{2}{1-r}}\,d\nu(\beta)\right)\leq 2 \log M, \quad \forall 0<r<1.
\end{align*}
Since the map 
\begin{align*}
	(0,r)\ni r \mapsto \frac{-r-1}{r} \log 	\left(\int\limits_{\Omega}\left\|A_\alpha h\right\|^{\frac{2}{r+1}}\,d\mu(\alpha)\right)+\frac{1-r}{r}\log \left(\int\limits_{\Delta}\left\|B_\beta h\right\|^{\frac{2}{1-r}}\,d\nu(\beta)\right) \in \mathbb{R}
\end{align*}
is differentiable, we can find its limiting value as $r\to 0$ using  L Hopital's rule. By doing so we get 
\begin{align*}
\int\limits_{\Omega}\|A_\alpha h\|^2\log \|A_\alpha h \|^2\, d\mu(\alpha)+\int\limits_{\Delta}\|B_\beta h\|^2\log \|B_\beta h \|^2\, d\nu(\beta)	\leq 2 \log M.
\end{align*}
By using the definition of entropy, we get the theorem.
\end{proof}
Motivated from the application of entropic uncertainty to finite groups by Krishna and Parthasarathy  \cite{KRISHNAPARTHASARATHY} and from the theory of group-frames for compact groups by Iverson \cite{IVERSON} we now give an application of Theorem \ref{UPG} to the compact groups. First we recall basic Peter-Weyl theory of unitary representations of compact groups. More details can be found in \cite{FOLLAND, SIMON}.

Let $G$ be a compact group. Let $\widehat{G}$ be the dual group of $G$ consisting of all non-equivalent irreducible unitary representations of $G$.  For $\pi \in \widehat{G}$, let $d_\pi$ be the dimension of the representation space of   $\pi$. For each $\pi \in \widehat{G}$, let $\{\pi_{i,j}\}_{1\leq i, j \leq d_\pi}$ be the matrix elements of $\pi$ in some orthonormal basis for its representation space. Following is the celebrated Peter-Weyl theorem.
\begin{remark}\cite{FOLLAND, SIMON}
Let $G$ be a compact group. The set 
\begin{align*}
\{\sqrt{d_\pi}\pi_{i,j}: 1\leq i, j \leq d_\pi, \pi \in \widehat{G}\}	
\end{align*}
is an orthonormal basis for $\mathcal{L}^2(G)$. 
\end{remark}
Let $H$ be any locally compact group with Haar measure $\nu$. Assume that there is a unitary representation $\rho: H \to \mathcal{B}(\mathcal{L}^2(G))$ (need not be irreducible)  and a   function $\phi \in \mathcal{L}^2(G)$  such that $\{\rho_h \phi\}_{h \in H}$ is a continuous Parseval frame for $\mathcal{L}^2(G)$  (such frames are known as group-frames, see \cite{IVERSON}). Define 
\begin{align*}
	&A_ {i,j,\pi}: \mathcal{L}^2(G) \ni f \mapsto d_\pi\langle f,  \pi_{i,j,\pi} \rangle \pi_{i,j,\pi} \in \mathcal{L}^2(G), \quad \forall 1\leq i, j \leq d_\pi, \pi \in \widehat{G},\\
	&B_h:\mathcal{L}^2(G) \ni f \mapsto \langle f, \rho_h \phi\rangle \rho_h \phi \in \mathcal{L}^2(G), \quad \forall h \in H.
\end{align*}
Let $f \in \mathcal{L}^2(G)_A\cap \mathcal{L}^2(G)_B$. 
Now by applying Theorem  \ref{UPG}, we get 
\begin{align*}
&-\sum_{1\leq i, j \leq d_\pi, \pi \in \widehat{G}}\|A_{i,j,\pi}f\|^2	\log \|A_{i,j,\pi}f \|^2-\int\limits_{H}\|B_h f\|^2\log \|B_hf \|^2\, d\nu(h)	\geq \\
&-2 \log \left(\displaystyle \sup_{1\leq i, j \leq d_\pi, \pi \in \widehat{G}, h \in H}\|B_h A_{i,j,\pi}^*\|\right).
\end{align*}
Using the definition of operators $A_ {i,j,\pi}$ and $B_h$  in the above equation gives the following corollary.
\begin{corollary}\label{C}
Under the set up defined earlier, we have 
\begin{align*}
	-&\sum_{1\leq i, j \leq d_\pi, \pi \in \widehat{G}}d_\pi|\langle f,  \pi_{i,j,\pi} \rangle|^2	\log  \left(d_\pi|\langle f,  \pi_{i,j,\pi} \rangle|^2\right)-\int\limits_{H} |\langle f, \rho_h \phi\rangle|^2\log | \langle f, \rho_h \phi\rangle|^2\, d\nu(h)	\geq \\
	&-2 \log \left(\displaystyle \sup_{1\leq i, j \leq d_\pi, \pi \in \widehat{G}, h \in H}\|B_h A_{i,j,\pi}\|\right).
\end{align*}
\end{corollary}

Note that Theorem \ref{UPG} works for arbitrary collections but for projections, Theorem \ref{KPU}  gives stronger lower bound (in the discrete case).
\section{Acknowledgments}
Author thanks Prof. B. V. Rajarama Bhat for several suggestions.\\
Author  thanks the anonymous reviewer for his/her   study of the manuscript and for a future direction of research on entropic uncertainty relations. The author is supported by Indian Statistical Institute, Bangalore, through the J. C. Bose Fellowship of Prof. B. V. Rajarama Bhat. He thanks Prof. B. V. Rajarama Bhat for the Post Doc position.

 \bibliographystyle{plain}
 \bibliography{reference.bib}

\end{document}